\documentclass{amsart}
\usepackage{amssymb, amsmath}
\usepackage{amsfonts}
\usepackage[all]{xy}
\numberwithin{equation}{section}
\newtheorem{lemma}{Lemma}
\newtheorem{corollary}{Corolarry}
\begin{document}
\title{HIDDEN SYMMETRIES AND $j(\tau)$}
\author{K.~M. Bugajska}
\address{Department of Mathematics and Statistics,
York University,
Toronto, ON, M3J 1P3}
\email{bugajska@yorku.ca}
\date{\today}
\begin{abstract}
We show that, for supersingular prime p, the image of a unique meromorphic function $G_p$ on $X_0(p)$ (of the degree two, with the polar divisor ${\{[0]_0,[\infty]_0\}}$)  under a certain Hecke operator is equal to $j(\tau)$ (up to some additive constant). This generates quantities of relations between the coefficients of $j(\tau)$ and leads to some group of hidden symmetries whose order must be divided by $p$.   
\end{abstract}


\maketitle

\section{INTRODUCTION}
Let $R_p$ denote a fundamental domain of $\Gamma_0(p)$ which for $p\geq{3}$ is given as
\begin{equation*}
R_{p}=R\cup{\bigcup}_{m=-\frac{p-1}{2}}^{\frac{p-1}{2}}\beta_{m}R,\quad \beta_{m}=ST^m,\quad S=\begin{pmatrix}0&1\\-1&0\end{pmatrix},T=\begin{pmatrix}1&1\\0&1\end{pmatrix}
\end{equation*}

Here $R$ is the standard quadrilateral fundamental domain for ${\Gamma}=SL_{2}\mathbb{Z}$ in the upper half-plane $\textit{H}$  i.e.
\begin{equation*}
R=\{\tau\in{\textit{H}}; |\tau|\geq{1}, {-1}\leq{Re\tau}\leq{0}\}\cup{\{\tau\in{\textit{H}}; |\tau|>1, 0<Re\tau<1\}}
\end{equation*}

We choose the domain $R_p$ because the imaginary axis is its symmertry   axis and this makes the visualization of the appropriate pairs of the Weierstrass points on $X_0(p)$ much easier. By $\mathcal{P}_{\mathcal{S}}$ we will denote the set of all supersingular primes $\{2,3,5,7,11,13,17,19,23,29,31,41,47,59,71\}$. They have the property that  the genus of $\textit{H}^{*}/{\Gamma^{+}_0(p)}$ is zero which implies that the modular curve $X_0(p)$  is hyperelliptic with the hyperelliptic involution given by the Friecke involution $\omega_{p}=\begin{pmatrix}0&1\\-p&0\end{pmatrix}$ acting on the upper halfplane $\textit{H}$ (i.e. $\widehat{\omega}_{p}={\frac{1}{\sqrt{p}}\omega_{p}}={\begin{pmatrix}0&{\frac{1}{\sqrt{p}}}\\-\sqrt{p}&0\end{pmatrix}}\in{SL_2{\mathbb{R}}}$). The curve $X_0(37)$ is also hyperelliptic however, it was shown by Ogg in  ~\cite{APO78},  that its hyperelliptic involution $v$ is an exceptional one i.e. it is not determine by $\widehat{\omega}_p$.

We start with showing that, for any prime $p\geq{3}$,  each stable point of ${\widetilde{\omega}_p}:{X_0(p)}\rightarrow{X_0(p)}$ determines an ideal class $\kappa$ in $\textit{Cl}\mathcal{O}_{\mathcal{K}}$ or in ${\textit{Cl}\mathcal{O}_{\mathcal{K}}}\cup{\textit{Cl}\mathcal{O}}$ (depending whether p is conguent to $1$ or to $3$ modulo $4$ respectively, ${\mathcal{K}={\mathbb{Q}(\sqrt{-p})}}$), and conversely, any such class $\kappa$ determines a unique ramification point of ${\pi_{\omega}}:{X_0(p)\rightarrow}{\Sigma}={X_0(p)/\widetilde{\omega}_p}$. Hence, the class numbers and the Gauss conjecture restricted to primes may be expressed (see $(5.4)$) in terms of the genus g of $X_0(p)$ and the genus $\gamma$ of $\Sigma$ using the Riemann-Hurwitz theorem. It makes possible to look for the properties of the class groups by investigating the properties of the ramification points of $\pi_{\omega}$ on $X_0(p)$. 

When ${\gamma}=0$, (i.e. when $p\in{\mathcal{P}_{\mathcal{S}}}$) then  the curve $X_0(p)$  must carry a unique meromorphic function $G_p$ that is invariant under $\widetilde{\omega}_p$, has polar divisor $G^{\infty}_{p}={\{[0]_0,[\infty]_0\}}$ and is normalized such that  its lifting $G_p(\tau)$  to $\textit{H}$ has the coefficient by $q^{-1}$ (in the Fourier expansion about $i\infty$) equal to one. When $X_0(p)$ has genus zero itself then $G_p(\tau)$ is expressed  in terms of the absolute invariant $\Phi_p(\tau)$, $(3.1)$, for $\Gamma_0(p)$. Next we use the Hecke operators $[\Gamma_0(p)\omega_{p}\Gamma]_0$, ~\cite{FD05}, to transform the $\Gamma^{+}_0(p)$-invariant functions $G_p(\tau)$ into $\Gamma$-invariant functions $P_p(\tau)$. It appears that, up to a constant given by the value $P_p(\rho)$, functions $j(\tau)$ and $P_p(\tau)$ coincide. Now, the construction of $P_p(\tau)$ out of $G_p(\tau)$ provides important relations between the coefficients $c_n$ of $j(\tau)$ which are investigated in sections $3$ and $4$. It seems that all this leads to some group of hidden symetries whose order must be divided by p.

On the other hand, the fact that all $c_n$'s are integers imply that for some supersingular primes p and for positive integers $m\notin{(p)}$ we may have pure non-rational (p,m)-absolute constants. When $\Omega^p_{n}={\Omega^p_m}$ for $m,n\notin(p)$, $m\neq{n}$ than we may have an absolute (not rational) $p$-invariant. (When genus of $X_0(p)$ is zero than all ${\Omega^p_n}=0$.)  

We easily notice that when we consider relations between the coefficients  for all singular primes simultaneously then the gigantic number of relations, their complexity  and character as well as the relation $(3.18)$ may indicate an occurrence of some vertex algebra.

At the end we list emerging open (to the author's knowledge) questions.

\section{WEIERSTRASS POINTS AND CLASS GROUPS}
Let $p$ be an odd prime. For any $\tau\in{\textit{H}}$ we denote  by $[\tau]_0$   the class $\Gamma_0(p)\tau$ and by $[\tau]$  the set $\Gamma\tau$.  Let ${\widetilde{\omega}_p}:{X_0(p)}\rightarrow{X_0(p)}$ be the map induced by the Friecke involution $\omega_p=\begin{pmatrix}0&1\\-p&0\end{pmatrix}$ acting on $\textit{H}$ and let ${\mathfrak{B}_p}=(b_1,\ldots,b_B)$ be the set of the stable points of of $\widetilde{\omega}_p$ and  hence the set of the ramification points of the mapping $\pi_{\omega}:{X_0(p)}\rightarrow\Sigma$ where $\Sigma\cong{\textit{H}/\Gamma^{+}_0(p)}\cong{X_0(p)/\widetilde{\omega}_p}$.

An element ${[\tau]_0}\in{X_0(p)}$ belongs to $\mathfrak{B}_p$ if and only if we have ${\omega_{p}\tau}=A\tau$ for some $A=\begin{pmatrix}a&b\\kp&d\end{pmatrix}\in{}\Gamma_0(p)$  (without a loss of generality we may assume that $a>0$).  Since ${\tau_0}=\frac{i}{\sqrt{p}}$ satisfies $\omega_{p}\tau_0=\tau_0$ we have ${[\tau_0]_0}\in{\mathfrak{B}_p}$. For other elements,  $\tau$ must be a solution of 
\begin{equation}
ap\tau^2+(bp+kp)\tau+d=0
\end{equation}

and hence we must have $A=\begin{pmatrix}a&b\\bp&d\end{pmatrix}$  and $\tau=\frac{1}{a}(\tau_0-b)$. When $b=1$ and $a=2$ we obtain ${\tau_1}=\frac{1}{2}(\tau_0-1)$. It represents the same element of $X_0(p)$ as the $\tau$ obtained when $b={-1}$. However for $a>2$ this is no longer true.  In this case we obtain the pair of elements ${\tau_{\pm}}=\frac{1}{a}(\tau_{0}\mp{b})$, $b>0$ (with corresponding matrices $A_{\pm}=\begin{pmatrix}a&\pm{b}\\{\pm}b&\frac{p{b^2}+1}{a}\end{pmatrix}$) such that their classes  $[\tau_{+}]_0$  and $[\tau_{-}]_0$ represent distinct points of $\mathfrak{B}_p$.

 Let $\mathcal{K}={\mathbb{Q}(\sqrt{-p})}$, let ${\mathcal{O}_{\mathcal{K}}}=[1,\beta_{-p}]$ denotes its ring of integers and let ${\mathcal{O}}=[1,2\beta_{-p}]$ denote its order with  conductor $2$, (here ${\beta_{-p}}=\sqrt{-p}$ when $p\equiv{1}mod{4}$ and $\beta_{-p}=\frac{1+\sqrt{-p}}{2}$ when $p\equiv{3}mod{4}$). Let $u$ denote the reduce number of $\tau$, that is, a unique element of the standard fundamental domain $R$ of $SL_{2}\mathbb{Z}$  such that $[u]=[\tau]$.
\begin{lemma}
For odd primes each ramification point ${[\tau]_0}\in{\mathfrak{B}_p}\subset{X_0(p)}$ corresponds to a unique class ${\kappa}\in{\textit{Cl}\mathcal{O}_{\mathcal{K}}}$ when $p\equiv{1}mod{4}$ or to ${\kappa}\in{\textit{Cl}\mathcal{O}_{\mathcal{K}}}\cup{\textit{Cl}\mathcal{O}}$ when $p\equiv{3}mod{4}$
\end{lemma}
\begin{proof}
First we notice that $[\tau_0]=[\sqrt{-p}]$ determines the class of the ideal $[1,\beta_{-p}]$ in $\textit{CL}\mathcal{O}_{\mathcal{K}}$ for $p\equiv{1}mod{4}$ and the class of $[1,2\beta_{-p}]$  in $\textit{Cl}\mathcal{O}$ for $p\equiv{3}mod{4}$.   Now ${u_0}={S\tau_0}=\sqrt{-p}$ is the reduced number representing this  class.  The class $[\tau_1]=[u_1]$ of $\Gamma$-equivalent elements ($u_{1}=\frac{1}{2}(u_{0}-1)$) always  determines a class in $\textit{Cl}\mathcal{O}_{\mathcal{K}}$. For other stable points of $\widetilde{\omega}_p$ we may use the fact that when the $gcd(ap,2bp,d)=1$ then the discriminant of a solution of $(2.1)$ is ${D(\tau)}=-4p$ and hence it coincides with the discriminant $\Delta_{\mathcal{K}}$ for $p\equiv{1}mod{4}$ or with the discriminant $\Delta_{\mathcal{O}}$ when $p\equiv{3}mod{4}$. When $gcd(ap,2bp,d)=l>1$ then we must have that $l=2$ and $D(\tau)=-p$. This occurs only when $p\equiv{3}mod{4}$ and then we have ${D(\tau)}={\Delta_{\mathcal{K}}}$.
\end{proof}

\begin{lemma}
Each class ${\kappa}\in{\textit{Cl}\mathcal{O}_{\mathcal{K}}}$ for $p\equiv{1}mod{4}$ and each class ${\kappa}\in{\textit{Cl}\mathcal{O}_{\mathcal{K}}}\cup{\textit{Cl}\mathcal{O}}$ for $p\equiv{3}mod{4}$ contains an  ideal ${\mathfrak{a}}=[1,\tau]$ such that ${\omega_{p}\tau}=A\tau$ for some $A\in{\Gamma_0(p)}$. Moreover,  all ideals $\mathfrak{a}$'s with the  property above which  represent a given  class $\kappa$ are given by the elements $\tau$'s which belong to a unique  $\Gamma_0(p)$-class $[\tau]_0$. 
\end{lemma}
\begin{proof}
We use the fact that the multiplicity of the class polynomial 
\begin{equation}
\mathcal{H}_p(X)=\prod_{{\kappa}\in{\textit{Cl}\mathcal{O}_{\mathcal{K}}}}(X-j(\kappa))\in{\mathbb{Z}[X]}
\end{equation}

in the diagonal modular polynomial $G_p(X)={F_p(X,X)}$ (where $F_p(X,j)=\prod_{k=0}(X-j\circ\alpha_k)$)  is equal to one.  The same is true when $p={3}mod{4}$ for 
\begin{equation}
\mathcal{H}_{\mathcal{O}}=\prod_{\kappa\in{\textit{Cl}\mathcal{O}}}(X-j(\kappa))\in{\mathbb{Z}[X]}
\end{equation}
This means that for the (unique) reduced number $u\in{R}$ representing a class $\kappa$ there exists the unique correspondence $\alpha\in{\Delta^{*}_p}=\bigcup_{k=0}^p{\Gamma}\alpha_k$ such that $\alpha{u}=u$. Here ${\alpha_k}=\begin{pmatrix}1&k\\0&p\end{pmatrix}$ when $k=0,\ldots,p-1$ and ${\alpha_p}=\begin{pmatrix}p&0\\0&1\end{pmatrix}$. 

To prove our statement we must show that there exists a unique alement ${\tau_m}={\beta_{m}{u}}\in{\mathfrak{R}_p}$, $m\in{\{-\frac{p-1}{2},\ldots,\frac{p-1}{2}\}}$, ${\beta_m}=\begin{pmatrix}0&1\\-1&-m\end{pmatrix}=ST^{m}$,  such that ${\omega_{p}\tau_m}={A_{m}\tau_m}$ for some $A_{m}\in{\Gamma_0(p)}$.  The unique correspondence  $\alpha\in{\Delta^{*}_p}$ such that ${\alpha{u}_0}=u_0$ is $\alpha={S\alpha_0}\in{\Gamma\alpha_0}$ and we have ${\tau_0}={\beta_{0}u_0}\in{\mathfrak{R}_p}$. Similarly it is easy to see that  the class in $\textsl{Cl}\mathcal{O}_{\mathcal{K}}$ represented by $[1,u_1]$  has the reduced number $u_{1}=\frac{1}{2}(u_0+1)$ that satisfies ${\alpha}u_{1}=u_1$ for unique correspondence ${\alpha}\in{\Gamma{\alpha_{\frac{p-1}{2}}}}$. Now ${\beta_{-\frac{p-1}{2}}}u_{1}={\omega_{p}\tau_1}\in{\mathfrak{R}_p}$ and  it determines ${[\tau_1]_0}\in{\mathfrak{B}_p}$. For the remaining cases, let an ideal $[1,u]$ with $u\in{R}$ represent a class  $\kappa$, let ${\alpha}={\gamma}\alpha_m$ for some $m$ and for some $\gamma\in\Gamma$, be the unique correspondence such that $\alpha{u}=u$.  For  each $j\in{\{-\frac{p-1}{2},\ldots,\frac{p-1}{2}\}}$ and  $\tau_{j}={\beta_{j}u}$  we have 
\begin{equation*}
\widetilde{\alpha}_j{\tau_j}=\tau_j \qquad with \qquad \widetilde{\alpha}_{j}=\beta_{j}\alpha\beta_j^{-1}\in{\Delta^{*}_p}
\end{equation*}

Since the multiplicity of the class polynomials in $G_p(X)$ are equal to one all $\widetilde{\alpha}_j$ belong to distinct cosets of $\Gamma$  in $\Delta^{*}_p$ and hence there exists only one $\tau_j$ such that $\widetilde{\alpha}_{j}\in{\Gamma{\alpha_p}}$  i.e. there is unique $\tau_j$ such that $\widetilde{\alpha}_{j}={\gamma^{-1}\omega_p}$ for some ${\gamma}\in{\Gamma}$. This means that 
\begin{equation}
\omega_{p}\tau_{j}={\gamma}\tau_j 
\end{equation}

We must show that ${\gamma}=\begin{pmatrix}a&b\\c&d\end{pmatrix}$ is in fact an element of $\Gamma_0(p)$. However, from $(2.4)$, $\tau_j$ is a solution of ${pa\tau^2+(pb+c)\tau+d}=0$.  Since ideals  $[1,u]$ and $[1,\tau_j]$ belong to the same class $\kappa$ we have that $D(\tau_j)$ is equal to $-4p$  or to $-p$  appropriately  and this implies that  ${\gamma}\in{\Gamma_0(p)}$. Hence,  each class ${\kappa}\in{\textit{Cl}\mathcal{O}_{\mathcal{K}}}$ for $p\equiv{1}mod{4}$ and each class ${\kappa}\in{\textit{Cl}\mathcal{O}_{\mathcal{K}}}\cup{\textit{Cl}\mathcal{O}}$ for $p\equiv{3}mod{4}$  is associated to a unique point of $X_{0}(p)$  that is stable under $\widetilde{\omega}_p$.
\end{proof}

\begin{corollary}
There is a one-one correspondence between the set $\mathfrak{B}_p$ of the ramification points of the projection $\pi_{\omega}:X_0(p)\rightarrow{\Sigma}\cong{X_0(p)/\widetilde{\omega}_p}$  and the set of  elements of $\textit{Cl}\mathcal{O}_{\mathcal{K}}$ for $p\equiv{1}mod{4}$ or the set of elements of $\textit{Cl}\mathcal{O}_{\mathcal{K}}\cup{\textit{Cl}\mathcal{O}}$ when $p\equiv{3}mod{4}$.
\end{corollary} 

This means that, when a prime p belongs to $\mathcal{P}_{\mathcal{S}}$, the number of elements in $\textit{Cl}\mathcal{O}_{\mathcal{K}}$ or in $\textit{Cl}\mathcal{O}_{\mathcal{K}}\cup{\textit{Cl}\mathcal{O}}$ is determined exactly by the number of the Weierstrass points of $X_0(p)$.  For  a general  prime p it is described by the Riemann-Hurwitz theorem (when the genus of $\Sigma$ is greater than zero). In other words we have
\begin{equation}
|\mathfrak{B}_{p}|=\textit{h}_{\mathcal{K}}+\textit{h}_2, \quad p\equiv{3}mod{4}\quad and \quad |\mathfrak{B}_p|=\textit{h}_{\mathcal{K}}, \quad p\equiv{1}mod{4}
\end{equation}

where $\textit{h}_{\mathcal{K}}=|\textit{Cl}\mathcal{O}_{\mathcal{K}}|$,  $\textit{h}_{2}=|\textit{Cl}\mathcal{O}|$ and $|\textsl{S}|$ denotes the cardinality of a set $\textsl{S}$. 

When the genus of ${\Sigma}=\textit{H}/\Gamma^{+}_0(p)$ is zero then the Riemann surface $X_0(p)$ is hyperelliptic with hyperelliptic involution corresponding to the Atkin-Lehner involution.  Thus all points of $\mathfrak{B}_p$ are exactly the $2g+2$ Weierstrass points on this surface. A set of representatives of these points, for $p\geq3$ contains $\tau_0$, $\tau_1$ and additional $2g$ points when the genus of $X_0(p)$ is $g\geq1$. We may choose these points as given by
\begin{equation}
\tau_{\pm}=\frac{1}{a}(\tau_{0}\mp{b})\qquad with \qquad 0<b<\frac{a}{2},\qquad 2<a<2\sqrt{\frac{p}{3}}
\end{equation}

and hence the reduced numbers associated to points of $\mathfrak{B}_p$ are given by
\begin{equation*}
u_{\pm}=\frac{1}{a}(u_{0}\pm1)\quad for\quad b=1, \qquad u_{\pm}=\frac{1}{a}(u_{0}\mp{r}) \quad for \quad a=br+1,\quad b\geq2
\end{equation*}
and
\begin{equation*}
u_{\pm}=\frac{1}{a}(u_{0}\pm{r}) \quad for \quad a=br-1
\end{equation*}

The genus of $X_0(P)$ is zero for $p=2,3,5,7,13$ and hence the set $\mathfrak{B}_p$ for these primes contains only two points: $\mathfrak{B}_{p}={\{[\tau_0]_{0},[\tau_1]_{0}\}}$ for $p\neq{2}$ and $\mathfrak{B}_{2}={\{[\frac{i+1}{2}]_{0},[\frac{i}{\sqrt{2}}]_{0}\}}$ (with the reduced numbers $i$ and $i\sqrt{2}$ respectively).  For other primes  in $\mathcal{P_{\mathcal{S}}}$ the  elements $\tau$'s representing the remaining $2g$, $\widetilde{\omega}_p$-stable  points of $X_0(p)$,  may be easily  found using the formulae $(2.6)$. It is easy to determine which of these $\tau$'s represent an ideal class in $\textit{Cl}_{\mathcal{K}}$ and which represents a class of $\textit{Cl}\mathcal{O}$. We collect some results for the odd primes $p\in{\mathcal{P}_{\mathcal{S}}}$ in the table above.
\begin{table}[t]
\begin{center}
\begin{tabular}{|c||c|c|c|c|c|c|c|c|c|c|c|c|c|c|}
\hline
$p$ & 3&5&7&11&13&17&19&23&29&31&41&47&59&71\\
\hline
$g(X_0(p))$ &0&0&0&1&0&1&1&2&2&2&3&4&5&6\\
\hline
$h_{\mathcal{K}}$ &1&2&1&1&2&4&1&3&6&3&8&5&3&7\\
\hline
$h_2$ &1&&1&3&&&3&3&&3&&5&9&7\\
\hline
$2g+2$ &2&2&2&4&2&4&4&6&6&6&8&10&12&14\\
\hline
\end{tabular}
\end{center}
\end{table}

\section{FUNCTION $P_p$}

\subsection{A genus zero case}

When $p=2,3,5,7,13$ the genus of the modular curve $X_0(p)$ is zero and the rational function field $\mathbb{C}(X_0(p))$ is given by the absolute invariant 
\begin{equation}
\Phi_p(\tau)=(\frac{\Delta(p\tau)}{\Delta(\tau)})^{\frac{1}{p-1}}=(\frac{\eta(p\tau)}{\eta(\tau)})^r, \qquad r=\frac{24}{p-1}
\end{equation}

that is $\mathbb{C}(X_0(p))=\mathbb{C}(\Phi_p(\tau))$. The absolute invariant $\Phi_p$ satisfies  $\Phi_p(\omega_{p}\tau)=p^{-\frac{r}{2}}\Phi_p^{-1}(\tau)$, ~\cite{JHS94},   and the  series expansions of these functions at $i\infty$ are:
\begin{equation}
\Phi_p(\tau)=q[1+\sum_{n=1}^{\infty}b_{n}q^n],\qquad b_{n}\in{\mathbb{Z}},\quad q=e^{2\pi{i}\tau} 
\end{equation}
and
\begin{equation}
\Phi_p(\omega_{p}\tau)=p^{-\frac{r}{2}}[q^{-1}+\sum_{n=0}^{\infty}a_{n}q^n], \qquad a_{n}\in{\mathbb{Z}}
\end{equation}

Two points of $X_0(p)$ that are stable under $\widetilde{\omega}_p$ are $[\tau_{0}={\frac{i}{\sqrt{p}}}]_0$ and $[\frac{1}{2}(\tau_0-1)]_0$ when  $p\neq{2}$ or $[\tau_{1}=\frac{1}{2}(i-1)]_0$ when $p=2$. Moreover we must have 
\begin{equation*}
\Phi_p(\tau_0)=-\Phi_p(\tau_1) \qquad with \quad \Phi_p(\tau_i)=\pm{p^{-\frac{r}{4}}},\quad i=0,1
\end{equation*}

Since the genus of $X_0(p)/\widetilde{\omega}_p$ is zero there exists a meromorphic function $G_p$ of  degree two on $X_0(p)$ with the polar divisor $\{{\infty}_{1},{\infty}_2\}={\{[0]_{0},[\infty]_0\}}$ such that the following diagram commutes.
\[
\xymatrix{
X_0(p)\ar[r]^{\widetilde{\omega}_p} \ar[d]^{G_p}
&X_0(p)\ar[ld]^{G_p}\\
\widehat{\mathbb{C}}&\\
Diag.1 &\\
}\]

The lifting $G_p(\tau)$ of the function $G_p$ to $\textit{H}$ must satisfy
\begin{equation}
G_p(\omega_{p}\tau)=G_p(\tau),\qquad G_p(\gamma\tau)=G_p(\tau),\quad \forall{\tau\in{\textit{H}}},\quad \forall{\gamma\in{\Gamma_0(p)}}
\end{equation}

so the function $G_p(\tau)$ is an automorphic function of $\Gamma^{+}_0(p)$.  The polar divisor of $G_p$  tells us that the series expansion of its lifting $G_p(\tau)$  about $i\infty$ is
\begin{equation}
G_p(\tau)=q^{-1}+G^{+}_p(q); \qquad G^{+}_p(q)=\sum_{n=0}^{\infty}\alpha^p_{n}q^n
\end{equation}

where we have normalized $G_p(\tau)$ such that the coefficient by $q^{-1}$ is one. We observe that the polar divisor of the function $\Phi_p(\tau)+\Phi_p(\omega_{p}\tau)$ is the same as the polar divisor of $G_p(\tau)$ and hence $(3.2)$, $(3.3)$ and $(3.5)$ imply that
\begin{equation}
G_p(\tau)= p^{\frac{r}{2}}[\Phi_p(\tau)+\Phi_p(\omega_{p}\tau)]=q^{-1}+\sum_{n=0}^{\infty}\alpha^p_{n}q^n
\end{equation}
with
\begin{equation}
\alpha_{0}=a_{0}=-b_1,\quad \alpha_{1}=p^{\frac{r}{2}}+a_1,\quad \alpha_{n}=p^{\frac{r}{2}}b_{n-1}+a_n,\quad n\geq{2}
\end{equation}

Now, any $\Gamma_0(p)$-automorphic function determines a $\Gamma$-automorphic function by applying an appropriate Hecke operator. In other words, we will  define the function $P_p(\tau)\in{\mathcal{A}_0(\Gamma)}=\mathbb{C}(j)$ as the image of $G_p(\tau)$ under the Hecke operator ~\cite{FD05}
\begin{equation}
[\Gamma_0(p)\omega_{p}\Gamma]_{0}: {\mathcal{A}_0(\Gamma_0(p))}\rightarrow{\mathcal{A}_0(\Gamma)}
\end{equation}
that is
\begin{equation}
P_p(\tau)=[\Gamma_0(p)\omega_{p}\Gamma]_{0}G_p(\tau)=\sum_{k=0}^{p}G_p(\omega_{p}\gamma_{k}\tau)=\sum_{k=0}^{p}G_p(\gamma_{k}\tau)
\end{equation}

where $\omega_{p}\gamma_{k}=\begin{pmatrix}-1&-k\\0&-p\end{pmatrix}$ for $k=0,1,\ldots,p-1$ and $\omega_{p}\gamma_{p}=\begin{pmatrix}0&-1\\p&0\end{pmatrix}$. We may rewrite the last formula as
\begin{equation}
P_p(\tau)=G_p(\tau)+\sum_{k=1}^{p-1}G_p(\frac{\tau+k}{p})
\end{equation}

Using expansions $(3.2)$ and $(3.3)$ we obtain
\begin{equation}
P_p(\tau)=j(\tau)-j(\widehat{u}_p)=q^{-1}+\alpha^p_0(p+1)+\sum_{n=1}^{\infty}(\alpha^p_n+p\alpha^p_{np})q^n
\end{equation}

where $\widehat{u}_p$ is the unique reduced number such that $j(\widehat{u}_p)=-P_p(\rho)$.  Thus we have 
\begin{equation}
P_p(\tau)=j(\tau)+P_p(\rho) \qquad and \quad P_p(\rho)=a^p_0(p-1)-744
\end{equation}

The vanishing of $P_p(\tau)$ at $\widehat{u}_{p}\in{R}\subset{R_p}$ implies that we have 
\begin{equation*}
G_p(\widehat{u}_p)=-\sum_{k=0}^{p-1}G_p(\frac{\widehat{u}_{p}+k}{p})=-\sum_{k=0}^{p-1}G_p(\gamma_{k}\widehat{u}_p)
\end{equation*}

That means that the value of $G_p$ at any one point of the  set of $p$ points on $X_0(p)$  determined by ${\{\beta_{k}\widehat{u}_p,k=\frac{-p+1}{2},\ldots,\frac{p-1}{2}\}}\subset{\mathfrak{R}_p}$ is exactly the negative sum of the values of $G_p$ at the remaining points.

  We observe that using $(3.10)$ and $(3.11)$ we may view the absolute $\Gamma$ invariant $j(\tau)$ on its standard fundamental domain $R$  as given (up to the additive constant ${P_p(\rho)}=-j(\widehat{u}_p)$) by the value of $G_p(\tau)$ at $\tau\in{R}$  plus the fundamental symmetric function $\sigma_1(G_p(\gamma_{0}\tau),\ldots,G_p(\gamma_{p-1}\tau))$ of the values of $G_p(\tau)$  at all points  that are $\Gamma$-equivalent to ${\tau}\in{R}$ but represent all the remaining, distinct $\Gamma_0(p)$-classes.

\subsection{A genus $g\geq{1}$ case}
We have already mentioned that when the genus of $X_0(p)$ is greater than two but the genus of $X_0(p)/\widetilde{\omega}_p$ is zero then $X_0(p)$ must be a hyperelliptic Riemann surface (we recall that any surface of genus two is already a   hyperelliptic one). The set $\mathfrak{W}_p$ of $2g+2$ Weierstrass points of $X_0(p)$ coincides with the set $\mathfrak{B}_p$ of the stable points of $\widetilde{\omega}_{p}:X_0(p)\rightarrow{X_0(p)}$ and with the set of elements  of $\textit{Cl}\mathcal{O}_{\mathcal{K}}$ when $p\equiv{1}mod{4}$ or with the set  $\textit{Cl}\mathcal{O}_{\mathcal{K}}\cup{\textit{Cl}\mathcal{O}}$ when  $p\equiv{3}mod{4}$. Since we have found all Weierstrass points  in the previous section we know that infinity is not one of them and hence we may introduce the function $G_p$ as follows.  Let $G_p$ be a unique (up to multiplicative constant) a degree two meromorphic function on $X_0(p)$ satisfying the properties of the commutative Diag.1 and whose polar divisor is ${\{\infty_1,\infty_2\}}={\{[0]_{0},[\infty]_0\}}$. Of course, its lifting $G_p(\tau)$ to $\textit{H}$ satisfies the conditions $(3.4)$  for $p=11,17,19,23,29,31,41,47,59,71$ and we normalize this function such that its Fourier expansion aroud $i\infty$ has a form
\begin{equation}
G_p(\tau)=q^{-1}+\sum_{n=0}^{\infty}\alpha^p_{n}q^n, \qquad q=e^{2\pi{i}\tau}
\end{equation} 

Similarly as in the genus zero case we use the Hecke operators $[\Gamma_0(p)\omega_{p}\Gamma]_0$ to transform $G_0(\tau)\in{\mathcal{A}_0(\Gamma_0(p))}$ into a $\Gamma$-invariant function
\begin{equation}
[\Gamma_0(p)\omega_{0}\Gamma]_{0}G_p(\tau)=\sum_{m=0}^{p}G_p(\omega_{p}\gamma_{m}\tau)\in{\mathcal{A}_0(\Gamma)}=\mathbb{C}(j)
\end{equation}

Similarly as before we will denote this image as $P_p(\tau)$. Thus we have
\begin{equation}
P_p(\tau)=G_p(\tau)+ \sum_{k={-\frac{p-1}{2}}}^{\frac{p-1}{2}}G_p(\beta_{k}\tau)=G_p(\tau)+\sum_{m=0}^{p-1}G_p(\frac{\tau+m}{p})
\end{equation}

whose Fourier expansion around $i\infty$ is 
\begin{equation}
P_p(\tau)= {q^{-1}+\alpha^p_{0}(p+1)+ \sum_{n=1}^{\infty}(\alpha^p_n+p\alpha^p_{np})q^n}\in\mathbb{C}(j)
\end{equation}

We see immediately that 
\begin{equation}
P_p(\tau)=j(\tau)+const \quad with\quad {const}=P_p(\rho)
\end{equation}

There exists unique element $\widehat{u}_p$ in the fundamental domain $R$ of $\Gamma$ such that $P_p(\rho)=-j(\widehat{u}_p)$ and hence we may rewrite the last formula as 
\begin{equation}
P_p(\tau)=j(\tau)-j(\widehat{u}_p)=j(\tau)+P_p(\rho)
\end{equation}

Hence
\begin{equation}
P_p(\widehat{u}_p)=0 \qquad and \quad P_p(\rho)=\alpha^p_{0}(p+1)-744 
\end{equation}

The first equality gives us 
\begin{equation}
G_p(\widehat{u}_p)=-\sum_{k={-\frac{p-1}{2}}}^{\frac{p-1}{2}}G_p(\beta_{k}\widehat{u}_p)
\end{equation}

that is, the value of $G_p(\tau)$ at $\widehat{u}_p$ is equal to the negative sum of the values of $G_p(\tau)$ at the remaining points of  $R_p-R$ that are $\Gamma$-equivalent  to $\widehat{u}_p$. Now, from $(3.16)$ and  $(3.18)$ we obtain
\begin{equation}
j(\tau)=P_p(\tau)+j(\widehat{u}_p)=q^{-1}+744+\sum_{n=1}^{\infty}c_{n}q^n
\end{equation}

and hence
\begin{equation}
744=\alpha^p_{0}(p-1)-P_p(\rho) \qquad and \quad c_{n}=\alpha^p_{n}+p\alpha^p_{np}
\end{equation}

\section{HIDDEN SYMMETRIES}
In the previous section we have found some relations between the coefficients $c_n$ of $j(\tau)$ (for $n\geq1$) and the coefficients $\alpha^p_n$ of the $\Gamma^{+}_0(p)$ invariant functions $G_p(\tau)$,   $p\in{\mathcal{P}_{\mathcal{S}}}$.  Namely, we have found that $c_n={\alpha^p_n+p{\alpha^p_{np}}}$. Suppose that we know all functions $G_p(\tau)={q^{-1}+\sum_{k=0}^{\infty}\alpha^p_k{q^k}}$ (i.e. we know all  $\alpha^p_k$'s). However, before we will investigate the consequences of these relations   let us look at the p-decompositions of the positive integers for a fixed prime p. Let ${n=\sum_{i=0}^N{a_i{p^i}}}\Leftrightarrow{(a_0,a_1,\ldots,a_N,0,\ldots)}$  with all coefficients ${a_i}\in{\{0,1,\ldots,p-1\}}$.  On $\mathbb{Z}_{\geq{0}}$ we may define the following two operations  $F_p$ and $\sigma_p$ as given by
\begin{equation*}
F_p(n)=pn\Leftrightarrow(0,a_0,\ldots,a_N,0,\ldots)\quad and \quad \sigma_p(n)=(\overline{a_0+1},a_1,\ldots,a_N,0,\ldots)
\end{equation*}

where ${\overline{a_0+1}}=(a_0+1)mod{p}$.  Thus, any positive integer n can be uniquely written as 
\begin{equation*}
n=\sigma_p^{a_0}\circ{F_p}\circ{\sigma_p^{a_1}}\circ{F_p}\circ{\ldots}\circ{F_p}\circ{\sigma_p^{a_N}}(0)
\end{equation*}

Let us fix $k>0, 0<l<p$. The action of the commutator $H^{k,l}_{p}:=[F_p^k,\sigma_p^l]$ on $n$ depends only on the value of $n{mod}p$. More precisely
\begin{equation*}
H^{k,l}_p(n)=-l+(\overline{a_0+l}-a_0)p^k
\end{equation*}

So, for example, for a positive $n\in(p)$ we have
\begin{equation*}
[F_p^k,\sigma^l_p](n)=l(p^{k}-1)=(l-1)p^{k}+(p-1){\sum_{m=0}^{k-1}p^m}\in{\mathbb{Z}^{+}}
\end{equation*}

and hence the operations $H^{k,l}_{p}$ transfer any positive element of the ideal $(p)$ into a positive integer out of the ideal. For   all remaining positive integers with $n{mod}{p}\neq{0}$ the expresion ${\overline{a_{0}+l}-a_0}$
is negative whenever ${a_{0}+l}\geq{p}$ and the $H^{k,l}_p$ image of such $n$ must also be  negative.   When $(a_{0}+l)<p$ then $n$ has a positive image $H^{k,l}_{p}(n)=l(p^{k}-1)$ which is exactly the same for all  integers $n$'s whose non-zero coefficient $a_0$ satisfies $(a_{0}+l)<p$. In all cases the image of $n\in{\mathbb{Z}^{+}}$ is not in $(p)$ i.e.  $H^{k,l}_p[\mathbb{Z}^{+}]\cap(p)=\emptyset$.

For each $n\in{\mathbb{Z}^{+}}$ the operations $F^k_p$  produce an infinite sequence $\mathcal{S}(n):=\{F^k_p(n), k\geq{0}\}$ (which we will write horizontally). If $n\notin{(p)}$ then the sequence $\mathcal{S}(n)$ is maximal but for any ${m=p^{r}n}\in(p)$ we have $\mathcal{S}(m)\subset{\mathcal{S}(n)}$. 

Operation $\sigma_p$ acting on $\mathbb{Z}_{\geq{0}}$  satisfies $\sigma_p^{p}=Id$. It connects the first terms of p distinct infinite horizontal $F_p$-sequences  and at least $p-1$ of them have to be maximal. For this reason we may view $\sigma_p$ as a transformation that moves a sequence $\mathcal{S}(n)$ with $n\Leftrightarrow(a_0,a_1,\ldots)$ into an $F_p$-sequence $\mathcal{S}(t)$ with $t\Leftrightarrow(\overline{a_{0}+1},a_1,\ldots,a_N,0,\ldots)$. In this way after p steps we will return to our original sequence $\mathcal{S}(n)$ (we will write such operation $\sigma_p$ vertically).  Thus we have arranged all non-negative integers in some sort of two dimensional ``step'' sequenses with the ``hight'' of each step equal to p. Actually, we merely expose the structure of the positive rational integers inside of the p-adic tree of the p-adic integers $\mathbb{Z}_p$.

Let us return to the coefficients $c_n$ of $j(\tau)$. The formulae $(3.11)$ and $(3.22)$ imply that for each $k\geq{1}$ we have
\begin{equation}
p^{k-1}c_{p^{k-1}n}-p^{k-2}c_{p^{k-2}n}+\ldots+(-1)^{k-1}c_{n}={p^k{\alpha^p_{p^k{n}}}+(-1)^{k-1}\alpha^p_n}
\end{equation}

This means that any infinite set of the coefficients $c_n$'s whose indexes belong to a horizontal, maximal $F_p$-sequence $\mathcal{S}(m)$ of integers are uniquely determined by the element $c_m$. We may introduce the operation $f_p$ which  corresponds to the operation $F_p$ as follows:
\begin{equation}
f_p:c_n\rightarrow{c_{pn}}, \qquad f_p(c_n)=c_{pn}=\frac{1}{p}c_n+(p\alpha^p_{p^2{n}}-\frac{\alpha^p_n}{p})
\end{equation} 

We see that for any $p\in{\mathcal{P}_{\mathcal{S}}}$ the coefficients of $j(\tau)$ form some kind of a graded $\mathbb{Z}/p\mathbb{Z}$ structure such that the operation $f_p$ produces an infinite, horizontal  $f_p$- sequences of coefficients whose all terms are determined by the first one.  On the contrary to $F_p$, the operation $\sigma_p$ does not induces any numerical formula between the coefficients $c_n$ and $c_{\sigma(n)}$. However, when we write $n=a_{0}+m$, $a_{0}={n}mod{p}$,  we may introduce the formal operation $\widehat{\sigma}_p$, $\widehat{\sigma}^p_{p}={Id}$, wich produces the following ``path'' of the length p of infinite $f_p$-sequences:
\begin{equation*}
\{f_p^k(c_n)\}_{k=0}^{\infty}\stackrel{\widehat{\sigma}_p}{\rightarrow}\{f_p^k(c_{(\overline{a_0+1}+m)})\}_{k=0}^{\infty}\stackrel{\widehat{\sigma}_p}{\rightarrow}\ldots\stackrel{\widehat{\sigma}_p}{\rightarrow}\{f^k_p(c_n)\}_{k=0}^{\infty}
\end{equation*}

Since $n=0$ does not generate an infinite $F_p$-sequence, $F_p(0)=0$, the same is true for $c_0$. This means that the role of the free coefficient $c_0$ in $j(\tau)$ is a distinguished one. In fact, the earlier obtained relations 
\begin{equation*}
c_{0}=\alpha^p_0(p-1)-P_p(\rho) \qquad and \quad j(\tau)=P_p(\tau)-P_p(\rho),\quad \tau\in{\textit{H}}
\end{equation*}

tell us about some connection between $c_0$ and the point ${\tau}={\rho}\in{\textit{H}}$ that is  associated to  each $p\in{\mathcal{P}_{\mathcal{S}}}$. It also tells us about a connection  between the whole function $j(\tau)$ and the torus determined by the lattice $L_{\rho}$. We have seen in ~\cite{KMA10} and ~\cite{KMB10} that this connection is very important  and a.o. manifests itself by the production of a generating matrix for the binary error correcting Golay code $\textit{G}_{24}$ out of the properties of the Veech modular curve $\textbf{T}^{*}={\textit{H}/\Gamma'}\cong{\mathbb{C}-L_{\rho}}/L_{\rho}$.

Summarizing, we see that while both operations $F_p$ and $\sigma_p$ on $\mathbb{Z}_{\geq{0}}$ are given by a concrete rules between integers, only the operation $f_p$ (inherited from $F_p$) is given by a concrete formula.  Each $p\in{\mathcal{P}_{\mathcal{S}}}$ arranges the set $\{c_{n}, n\geq{0}\}$ as  step, graded sequences whose all horizontal terms are uniquely determined by the first coefficient  of the sequence, whose steps have lenght equal to $p$ and (for a fixed p) are not related to each other by any numerical formula. Moreover, the coefficient $c_0$ possess a distinquish role in this presentation.

We have seen that for all primes $p\in{\mathcal{P}_{\mathcal{S}}}$ the image $P_p(\tau)$ of each $G_p(\tau)$ under the Hecke operator $[\Gamma_0(p)\omega_{p}\Gamma]_0$ coincide with the sum of the absolute invariant $j(\tau)$ and the value of the function $P_p(\tau)$ at the distinguish point $\rho$. This determines the unique point ${[\widehat{u}_p]_0}\in{X_0(p)}$,  ${\widehat{u}_p}\in{R}\subset{R_p}$, that has the property 
\begin{equation*}
G_p(\widehat{u}_p)=-\sum_{k=0}^{p-1}G_p(\frac{\widehat{u}_{p}+k}{p})
\end{equation*} 

Moreover, for each such  prime (and only for those primes) we may write the Klein modular function $j(\tau)$ as 
\begin{equation}
j(\tau)=G_p(\tau)+\sigma_1(\{G_p(\frac{\tau+k}{p}),k=0,\ldots,p-1\})+j(\widehat{u}_p)
\end{equation}

where $\sigma_1$ is the standard symmetric function of $p$ variables. Let us rewrite the last formula $(4.3)$ as follows
\begin{equation}
j(\tau)=G_p(\tau)+\sum_{r\in{\mathcal{C}_p}}G_p(\frac{\tau}{p}+r)+const,\qquad \tau\in{\textit{H}}
\end{equation}

where $\mathcal{C}_{p}=\left\langle \frac{1}{p}+L_{\frac{\tau}{p}}\right\rangle$ is a cyclic subgroup of $p-points$ of the lattice $[1,\frac{\tau}{p}]$. This illustrates the fact that the modular invariant $j(\tau)$ has a hidden, p-cyclic symmetries wich are associated to the modular pairs $(L_{\frac{\tau}{p}},\mathcal{C}_p)$ stabilized by $\Gamma_0(p)$. So, the expression of the Klein modular function $j(\tau)$ in terms of the $\Gamma^{+}_0(p)$ invariant functions $G_p(\tau)$  using the appropriate Hecke operators $[\Gamma_0(p)\omega_{p}\Gamma]_0$ exhibits a hidden p-cyclic symmetry which occurs in both, in the formula $(4.4)$ and in the arrangement of the coefficients $c_n$'s as p-step, graded sequences. 

We notice that the formula $(3.18)$ for $P_p(\tau)$ immediately relates this function to the denominator identities of the Monster Lie algebra discovered by Borcherds, ~\cite{REB92}, and others. These   facts together with the relations (obtained when we consider all $p\in{\mathcal{P}_{\mathcal{S}}}$ simultaneously) between the coefficients $c_{n}\in{\{f^k_p(c_m)\}_{k=0}^{\infty}}$ and $c_{\sigma^l_{p}n}\in{\{f^k_{p_1}(c_{m_1})\}_{k=0}^{\infty}}$ for $p\neq{p_1}$ indicates that the representation of the full hidden symmetry of $j(\tau)$ may be given by some sort of a vertex operator algebra.

\section{CONCLUSIONS}
In all three papers, in ~\cite{KMA10}, ~\cite{KMB10} and in the present one we observe a particular relationship between the modular curve $Y_0(1)={\textit{H}/\Gamma}$ and the curve $\mathcal{E}:t^{2}=4u^{3}-1$ analytically isomorphic to the Veech modular curve $\textbf{T}^{*}={\textit{H}/\Gamma'}\cong{\mathbb{C}-L_0}/L_0$ with $L_{0}\cong{L_{\rho}}$, ${\Gamma'}=[\Gamma,\Gamma]$. As we have already mentioned, the properties of the Veech modular curve $\textbf{T}^{*}$  produce a generating matrix for the Golay code $\textit{G}_{24}$ whose full automorphism group is given by the Matieu group $\mathcal{M}_{24}$. Hence, $\mathcal{M}_{24}$ must be a subgroup of the full group of hidden symmetries associated to the function $j(\tau)$. In this paper we indicate that for each $p\in{\mathcal{P}_{\mathcal{S}}}$ function $j(\tau)$ exhibits a hidden p-cyclic symmetry coming from the relations between $j(\tau)$ and $\Gamma^{+}_0(p)$-invariant functions $G_p(\tau)$'s and described in the previous section. Hence, the the full group of hidden symmetries associated to $j(\tau)$ contains $\mathcal{M}_{24}$ and has the order wich is devided by all $p\in{\mathcal{P}_{\mathcal{S}}}$.  If this group is a simple one it has to be the monster group $\textbf{M}$.  

The particular role of the point ${\rho}\in{\textit{H}}$ revealed in $(3.12)$ and $(3.18)$  is supported by the mentioned above results of ~\cite{KMA10}, ~\cite{KMB10} as well as by the results of Harada and Lang in ~\cite{HL89}. They have shown that the all five curves associated to some special conjugacy classes of the Conway group $.O$ of all automorphisms of the Leech lattice $\Lambda$ (which,by the way, originates with $\textit{G}_{24}$) are elliptic curves that represent the Riemann surface $\textbf{T}^{*}$, that is, all these curves have the j-invariant  equal to zero.

In our considerations we have assumed that all meromorphic functions $G_p$ on $X_0(p)$ are known and we have used them to describe relationships between the coefficients $c_n$, $n\in{\mathbb{Z}^{+}}$ of $j(\tau)$. However, as to the author's knowledge, except for $p=2,3,5,7$ and $13$, when $G_p(\tau)$ can be written in terms of the absolute invariants $\Phi_p$ as in $(3.6)$, it is not a case. So, let us reverse the situation. Let us try to find properties of the coefficients $\alpha^p_n$'s for $p=11,17,\ldots,71$ and $n\in{\mathbb{Z}^{+}}$. Since $c_{n}\in{\mathbb{Z}^{+}}$ we may assume  that $\alpha^p_{n}={\Omega^p_n+g^p_n}$ where $g^p_{n}\in{\mathbb{Q}}$ and $\Omega^p_n$ is either zero or pure not-rational. When the genus of $X_0(p)$ is zero than all ${\Omega^p_n}=0$ and all $\alpha^p_n$ are integers.  For the remaining $p\in{\mathcal{P}_{\mathcal{S}}}$ it may not be so. However, we immediatelly obtain that
\begin{equation}
\Omega^p_{pn}=-\frac{\Omega^p_n}{p} \qquad and \quad \Omega^p_{p^{k}n}={(-1)^k}\frac{\Omega^p_n}{p^k}
\end{equation}

Hence, for any $n=p^{l}m$ with $p\nmid{m}$ i.e. for any $n\in{\{F^k_p(m)\}_{k=0}^{\infty}}$ we have
\begin{equation}
\Omega^p_{p^{l}m}=(-1)^{l}\frac{\Omega^p_m}{p^l}\rightarrow{0} \qquad as \quad l\rightarrow{\infty}
\end{equation} 

If ${\Omega^p_m}\neq{0}$, ( $p\nmid{m}$) then we will call it the universal $(p,m)$-invariant or the universal $(p,m)$-constant. Now, what can we find about the pure rational part $g^p_n$ of $\alpha^p_n$?  Let $g^p_{n}=\frac{a_n}{b_n}$ with $gcd(a_n,b_n)=1$  and with the obvious upper index $p$ omitted. From $(3.22)$ we obtain $a_{n}b_{pn}+pa_{pn}b_{n}=b_{n}b_{pn}c_n$  and  hence $b_n|b_{pn}$  and  $b_{pn}|pb_n$. For a more general case let us rewrite the formulw $(4.1)$ as
\begin{equation}
p^{k}\alpha^p_{p^{k}n}+(-1)^{k+1}\alpha^p_{n}=p^{k-1}c_{p^{k-1}n}-p^{k-2}c_{p^{k-2}n}+\ldots+(-1)^{k-1}c_n
\end{equation} 

and let us denote the left side of of this equality by $A^{p,k}_n$. Since all of the $\Omega$-terms will cancel with each other we will end with
\begin{equation*}
a_{n}b_{p^{k}n}+(-1)^{k+1}p^{k}b_{n}a_{p^{k}n}=b_{n}b_{p^{k}n}A^{p,k}_n
\end{equation*}

Similarly as before, for any $k\geq{1}$, we must have $b_{n}|b_{p^{k}n}$ and $b_{p^{k}n}|p^{k}b_n$. There are (a.o.) two simple possibilities: either we have $b_{pn}={pb_n}$ for $n\in{\mathbb{Z}^{+}}$ (and hence $b_{p^{k}n}=p^{k}b_n$)  or we have $b_{pn}=b_n$. In the former case we would obtain 
\begin{equation*}
G_p(\tau)=q^{-1}+(\Omega^p_{0}+g^p_0)+\sum_{n\notin{(p)}}{\sum_{k=0}^{\infty}[\frac{\Omega^p_n}{p^k}+\frac{a_{p^{k}n}}{p^{k}b_n}]q^{p^{k}n}}
\end{equation*}

and in the latter one
\begin{equation*}
G_p(\tau)={q^{-1}+(\Omega^p_{0}+g^p_0)+\sum_{n\notin{(p)}}{\sum_{k=0}^{\infty}[\frac{\Omega^p_n}{p^{k}}+\frac{a_{p^{k}n}}{b_n}]q^{p^{k}n}}}
\end{equation*}

which in a case when all $\Omega^p_n$ are equal to zero becomes
\begin{equation*}
G_p(\tau)={q^{-1}+g^p_{0}+\sum_{n\notin{(p)}}\frac{1}{b_n}{\sum_{k=0}^{\infty}a_{p^{k}n}q^{p^{k}n}}}
\end{equation*}

In principle, there is nothing in the way for  both of the above cases to occur. We may also have a mixed situation when , for example, $b_{p^{3}n}=pb_{p^{2}n}$ and $b_{p^{2}n}=b_{pn}=b_n$, etc. (All of the division conditions stated above can be, in such mixed cases,  fullfiled as well.) 

Let us say something about  the section $2$.  We have shown that the class number $\textit{h}_{\mathcal{K}}$ for $p\equiv{1}mod{4}$ and $\textit{h}_{\mathcal{K}}+\textit{h}_2$ for $p\equiv{3}mod{4}$ is given by the total branch number B of the mapping $\pi_{\omega}: X_0(p)\rightarrow{\Sigma}\cong{X_0(p)/\widetilde{\omega}_p}$. This means that it is ruled by the Riemann-Hurwitz formula which, in our case, becomes $g={2\gamma-1+\frac{B}{2}}$ where g is the genus of $X_0(p)$,  $\gamma$ denotes the genus of $\Sigma$ and $B={\sum_{x\in{X_0(p)}}b(x)}$ is always even. (For any $x\in{X_0(p)}$ the ramification number is $r(x)=b(x)+1$ with $\sum_{x\in{\pi^{-1}_{\omega}(\varsigma)}}r(x)=2$ for each $\varsigma\in{\Sigma}$). Thus we have
\begin{equation}
\textit{h}_{\mathcal{K}}={2g+2-4\gamma} \quad and \quad {\textit{h}_{\mathcal{K}}}+{\textit{h}_2}={2g+2-4\gamma}
\end{equation}

for $p\equiv{1}mod{4}$ and for $p\equiv{3}mod{4}$ respectively. For a prime $p\in{\mathcal{P}_{\mathcal{S}}}$ we have ${\gamma}=0$ and the total branching number B is equal to the cardinality $2g+2$ of the set $\mathfrak{W}_p$ of the Weierstrass points of $X_0(p)$.  

The Gauss conjecture restricted to the class numbers for imaginary quadratic fields $\mathcal{K}={\mathbb{Q}(\sqrt{-p})}$ states that
\begin{equation}
\textit{h}_{\mathcal{K}}\rightarrow{\infty} \qquad as \quad  p\rightarrow{\infty}
\end{equation}  

and we have found a geometric interpretation of this fact given by the number of the ramification points of the projections ${{\pi}_{\omega}}$'s. Since in this case the whole class groups $\textit{Cl}_{\mathcal{K}}$ are the principal genus (which tends to be cyclic more often that not ~\cite{HI04}) we may approach the open question when $\textit{Cl}_{\mathcal{K}}$ is cyclic by looking for the properties of the ramification points of ${\pi}_{\omega}$ on $X_0(p)$.

Finally let us list some questions (which may already have nice and simple answers)
\begin{itemize}
\item Find the values of $P_p(\rho)$ for $p\in{\mathcal{P}_{\mathcal{S}}}$
\item Find $\widehat{u}_{p}\in{R}$
\item Use the relations between the coefficients $c_n$ of $j(\tau)$ and $\alpha^p_n$ of $G_p(\tau)$ for all $p\in{\mathcal{P}_{\mathcal{S}}}$ (determined in this paper) to find whether this leads to some vertex operator algebra.
\item Find whether $\Omega^p_{n}=0$ for all $p\in{\mathcal{P}_{\mathcal{S}}}$
\item If $\Omega^p_{n}\neq{0}$ determine whether $\Omega^p_{n}=\Omega^p_m$ for $n,m\notin(p)$, $n\neq{m}$. If p would have such property then we would have (pure not rational) universal p-constant $\Omega^{p}\notin{\mathbb{Q}}$.
\item Find whether $b_{n}\neq{1}$ for $n\notin(p)$
\item Find when the class group $\textit{Cl}_{\mathcal{K}}$, ${\mathcal{K}}={\mathbb{Q}(\sqrt{-p})}$,   is a cyclic one using the properties of the stable points of $\widetilde{\omega}_p$ on $X_0(p)$.
\item Express the class number $\textit{h}_{\mathcal{K}}$ for an arbitrary imaginary quadratic field $\mathcal{K}={\mathbb{Q}(\sqrt{-N})}$, $N>0$, using the Ogg's results and the Riemann-Hurwitz theorem for the appropriate projections and describe the group properties of $\textit{Cl}_{\mathcal{K}}$ using the properties of the ramification points (of these projections) on $X_0(N)$. 
\end{itemize}

\end{document}